\documentclass[12pt]{article}
\usepackage{amssymb,amsmath,graphicx,verbatim,eufrak,amsthm}
\usepackage{color}
\newtheorem{thm}{Theorem}
\newtheorem{cor}[thm]{Corollary}
\newtheorem{lem}[thm]{Lemma}
\newtheorem{prop}[thm]{Proposition}

\newtheorem*{thm*}{Theorem}

\theoremstyle{definition}
\newtheorem{dfn}[thm]{Definition}
\theoremstyle{remark}
\newtheorem{rem}[thm]{Remark}
\numberwithin{equation}{section}

\newcommand{\norm}[1]{\left\Vert#1\right\Vert}

\newcommand{\set}[1]{\left\{#1\right\}}

\newcommand{\br}[1]{\left[#1\right]}
\newcommand{\Br}[2]{ \left[#1 \ \big| \ #2 \right] }
\newcommand{\sr}[1]{\left(#1\right)}

\newcommand{\Integer}{\mathbb{Z}}

\newcommand{\Z}{\Integer}
\newcommand{\N}{\mathbb{N}}

\DeclareMathOperator{\E}{\mathbb{E}}     

\renewcommand{\Pr}{}
\let\Pr\relax
\DeclareMathOperator{\Pr}{\mathbb{P}}

\newcommand{\1}[1]{\mathbf{1}_{\set{ #1 } }}
\newcommand{\ov}[1]{\overline{#1}}

\newcommand{\p}{\partial}

\newcommand{\F}{\mathcal{F}}
\newcommand{\dist}{\mathrm{dist}}

\begin{document}

\title{Entropy of Random Walk Range}

\author{Itai Benjamini\thanks{Faculty of Mathematics and Computer Science, The
Weizmann Institute of Science, Rehovot 76100, Israel. Email:
\texttt{\{itai.benjamini,gady.kozma,ariel.yadin\}@weizmann.ac.il}.} \and Gady
Kozma\footnotemark[1]
\and Ariel
Yadin\footnotemark[1]
\and Amir
Yehudayoff\thanks{School of Mathematics, Institute for Advanced
Study, Princeton NJ, 08540. Email:
\texttt{amir.yehudayoff@gmail.com}. Partially supported by NSF
grant CCF 0832797.} }

\date{ }

\maketitle

\begin{abstract}
We study the entropy of the set traced by an $n$-step random walk on $\Z^d$.
We show that for $d \geq 3$, the entropy is of order $n$.
For $d = 2$, the entropy is of order $n/\log^2 n$.
These values are essentially governed by the size of the boundary of the trace.
\end{abstract}

\section{Introduction}

A natural observable of a random walk is its \emph{range},
the set of positions it visited. In this note we study the entropy of this range --
roughly, how many bits of information are needed in order to describe it.
We calculate the entropy of the range of a random walk on $\Z^d$, $d \in \N$, up to constant factors.

\subsection{Main Result}

Let $S(0),\ldots,S(n)$ be a simple symmetric nearest-neighbor random walk on $\Z^d$, $d \in \N$,
of length $n$.
Define the \emph{range} of the random walk to be
$$ R(n) = \set{ S(0), S(1), \ldots, S(n) } , $$
the set of vertices visited by the walk.

In this note we study the entropy of $R(n)$ as a function of $n$
(for formal definition of entropy, see Section~\ref{sec: ent}).
We calculate the value of the entropy, $H(R(n))$, up to constants, precisely:

\begin{thm}
\label{thm: main}
For $d = 2$ there exist constants $c_2, C_2 > 0$ such that for all $n \in \N$,
$$ c_2 \frac{n}{\log^2(n)} \leq H(R(n)) \leq C_2 \frac{n}{\log^2(n)},$$
and for $d \geq 3$ there exist constants $c_d, C_d > 0$ such that for all $n \in \N$,
$$c_d n \leq H(R(n)) \leq C_d n.$$
\end{thm}

The proof of Theorem~\ref{thm: main} is organized as follows: we
first prove the lower bound which is easier and follows directly
from estimates on the size of the boundary of the range; in two
dimensions the boundary of the range of the walk is of order
$n/\log^2 n$, and in higher dimensions it is linear in $n$. This
is done in Section~\ref{sec: lower}. We then show the upper bound
which requires a certain renormalization argument. An interesting
feature of the procedure is that at each step of the
renormalization process, the number of ``active'' boxes is not
determined by examining the previous renormalization step, but
rather globally. This is done in Section~\ref{sec: upper}.

The one dimensional case is not difficult.

\paragraph{Exercise.} In the case $d=1$,
there exist constants $c_1,C_1 >0$ such that for all $n \in \N$,
$$ c_1 \log n \leq H(R(n)) \leq C_1 \log n . $$

\paragraph{Acknowledgements.}  We thank Eric Shellef for useful discussions.  We also thank Elchanan Mossel for his
help with the construction in Section \ref{scn: extracting entropy}.

\section{Entropy of Random Walk}

\subsection{Entropy}
\label{sec: ent}

Here we provide some background on entropy.
Let $X$ be a random variable taking values in an arbitrary finite set $\Omega$.
For $x \in \Omega$, let $p(x)$ be the probability that $X = x$.
The \emph{entropy} of $X$ is defined as $H(X) = \E [ - \log p(X) ]$
(all logarithms in this note are base $2$).
For two random variables $X$ and $Y$,
the \emph{conditional entropy} of $X$ conditioned on $Y$ is defined as $H(X|Y) = H(X,Y) - H(Y)$.

\begin{prop} \label{prop: properties of entropy}
The following relations hold:
\begin{enumerate}
\item \label{item: H leq log} $H(X) \leq \log |\Omega|$.
\item \label{item: H(f(X)|X)} For every function $f$, $H(f(X)|X) = 0$.
\item \label{item: chain cond. entropy} $H(X) \leq H(Y) + H(X|Y)$.
\end{enumerate}
\end{prop}

For more information on entropy and for proofs of these properties see,
e.g., \cite[Chapter 2]{InformationTheory}.

\subsection{Lower Bound}
\label{sec: lower}
\paragraph{Notation.}
By $\Pr_z$ and $\E_z$ we denote the probability measure and expectation of the random walk
conditioning on $S(0)=z$.  We denote $\Pr = \Pr_0$ and $\E = \E_0$.
Let $z,w \in \Z^d$ and $A \subset \Z^d$.
Denote by $\dist(z,w)$ the graph distance between $z$ and $w$ in $\Z^d$.
Denote $\dist(z,A) = \inf \set{ \dist(z,a) \ : \ a\in A }$.
We write $z \sim w$ if $\dist(z,w) = 1$, and $z \sim A$ if $\dist(z,A) = 1$.
The {\em inner boundary} of $A$ is defined as
$$ \p A = \set{ z \in A \ : \ z \sim \Z^d \setminus A } . $$
Let $p_n(A) = \Pr [ R(n) = A ]$.

\begin{lem} \label{lem: upper bound on p(A)}
For every  $A \subset \Z^d$,
$$ p_n(A) \leq \Big(1-\frac{1}{2d} \Big)^{|\p A|-1} . $$
\end{lem}

\begin{proof}
Let $T_0 = 0$ and define inductively for $j \geq 1$,
$$ T_j = \inf \set{ t \geq T_{j-1} + 1 \ : \ S(t) \in \p A } . $$
By the strong Markov property, for any $j < |\p A|$,
$$ \Pr \Br{ S(T_j+1) \not\in A }{ S(0),\ldots,S(T_j), T_j < \infty } \geq \frac{1}{2d} . $$
The event $A \subseteq R(n)$ implies that $T_j \leq n$ for all $j \leq |\p A|$.
The event $R(n) \subseteq A$ implies that $S(T_j+1) \in A$ for all $j \leq |\p A|-1$.
Let $E_j$ be the event that $S(T_j+1) \in A$ and $T_{j+1} < \infty$. Thus,
\begin{align*}
  \Pr \br{ R(n) = A } & \leq \Pr \Big[ \bigcap_{j=1}^{|\p A|-1} E_j \Big] \\
  & \leq \prod_{j=1}^{|\p A|-1} \Pr \Br{ E_j }{ E_1, \ldots, E_{j-1} } \leq
 \Big(1-\frac{1}{2d} \Big)^{|\p A|-1}  .
  \qquad \qedhere
\end{align*}
\end{proof}

Lemma \ref{lem: upper bound on p(A)} shows that in order to lower bound the entropy
of the random walk trace it is enough to lower bound the expected value of the size of the inner boundary
of the random walk trace.

\begin{cor} \label{cor: lower bound H(R(n))}
$H(R(n)) \geq - \log \Big(1- \frac{1}{2d} \Big) \cdot \E [ |\p R(n)|-1 ]$.
\end{cor}

The following lemma gives the lower bound for the entropy of the random walk trace.

\begin{lem} \label{lem: entropy lower bound}
For any $d \geq 2$, there exists a constant $c_d>0$ such that for all $n \in \N$,
$$ H(R(n)) \geq \left\{ \begin{array}{lr}
    c_2 \frac{n}{\log^2(n)} & d = 2 , \\
    c_d n & d \geq 3 .
\end{array} \right. $$
\end{lem}

\begin{proof}
By Corollary \ref{cor: lower bound H(R(n))}, it suffices to show that
$$ \E [|\p R(n)| ] \geq \left\{ \begin{array}{lr}
    c_2 \frac{n}{\log^2(n)} & d = 2 , \\
    c_d n & d \geq 3 ,
\end{array} \right. $$
for some constants $c_d > 0$.
For $z \in \Z^d$, define
$T_z = \inf \set{ t \geq 0 \ : \ S(t) = z }$.
By Lemma~19.1 of \cite{Revesz}, and by the transience of the random walk for $d \geq 3$,
there exist constants $c_d > 0$ such that for any $z \sim w \in \Z^d$,
$$ \Pr_z [ T_w > n ] \geq \left\{ \begin{array}{lr}
    \frac{c_2}{\log n} & d = 2, \\
     c_d & d \geq 3 .
\end{array} \right. $$
Denote the right-hand side of the above equality by $f_d(n)$.
Using the strong Markov property at time $T_z$, for any $z \sim w \in \Z^d$,
$$ \Pr \br{ z \in \p R(n) } \geq
\Pr \br{ T_z \leq n \ , \ T_w > n } \geq
f_d(n) \Pr \br{ T_z \leq n } . $$
This proves the lemma, since
$$ \E [ |\p R(n) | ] \geq f_d(n) \sum_{z \in \Z^d} \Pr [ T_z \leq n ] = f_d(n) \E [ |R(n)| ] , $$
and since
$$ \E [ |R(n)| ] \geq \left\{ \begin{array}{lr}
    c'_2 \cdot \frac{n}{\log n} & d = 2 , \\
    c'_d n & d \geq 3 ,
\end{array} \right. $$
for some constants $c'_d > 0$ (see, e.g., Theorem 20.1 in \cite{Revesz}).
\end{proof}

\subsection{Upper Bound}
\label{sec: upper}

We now show that the lower bounds on the entropy of the random walk trace
given by Lemma \ref{lem: entropy lower bound} are correct up to a constant.
The transient case is much simpler than the two-dimensional case.

\begin{prop}
For $d \geq 3$, there exists a constant $C_d > 0$ such that for all $n \in \N$,
$$ H(R(n)) \leq C_d \cdot n . $$
\end{prop}

\begin{proof}
Let $\Omega = \set{ A \subset \Z^d \ : \ p_n(A) > 0 }$.
By clause \ref{item: H leq log} of
Proposition \ref{prop: properties of entropy}
it suffices to prove that $|\Omega| \leq  (2d)^n$.
This follows from the fact that the number of
possible $n$-step trajectories in $\Z^d$ starting at $0$ is $(2d)^n$.
\end{proof}

\subsection{Two Dimensions}

We now turn to the two-dimensional case, which is more elaborate.

For $z \in \Z^2$, we denote by $\norm{z}$ the $L^2$-norm of $z$.
Denote $$ T_{z,r} = \inf \set{ t \geq 0 \ : \ \norm{S(t)-z} \leq r } , $$
and denote $T_r = T_{\vec{0},r}$.
Also denote
$$ \tau_{z,r} = \inf \set{ t \geq 0 \ : \ \norm{S(t)-z} \geq r } , $$
and denote $\tau_r = \tau_{\vec{0},r}$.

\subsubsection{Probability Estimates}

We begin with some classical probability estimates regarding the random walk on $\Z^2$,
which we include for the sake of completeness.

\begin{lem} \label{lem: maximum of E[RW]}
There exists a constant $C>0$ such that for all $n \in \N$,
$$ \E \Big[ \max_{0 \leq k \leq n} \norm{S(k)}^2 \Big]  \leq C n  . $$
\end{lem}

\begin{proof}
Let $S(k) = (X(k),Y(k))$,
so $\norm{S(k)}^2 = |X(k)|^2 + |Y(k)|^2$.
Doob's maximal inequality (see, e.g., \cite[Chapter II]{RY}) on the martingale $X(k)$ tells us that
$$ \E \Big[  \max_{0 \leq k \leq n} |X(k)|^2 \Big] \leq  4 \E \br{ |X(n)|^2 } . $$
The martingale $|X(k)|^2 - k/2$ tells us that $\E \br{ |X(n)|^2 } = n/2$, which completes the proof,
since $X(k)$ and $Y(k)$ have the same distribution.
\end{proof}

\begin{lem} \label{lem: maximum of RW}
There exist constants $c_1, c_2 > 0$ such that for all $n \in \N$ and $\lambda > 0$,
\begin{align*}
\Pr \Big[ \max_{1 \leq j \leq n} \norm{S(j)} \geq \lambda \Big]
& \leq c_1 \cdot \exp \sr{ - c_2 \frac{\lambda^2}{n} } .
\end{align*}
\end{lem}

\begin{proof}
This is a consequence of Theorem 2.13 in \cite{Revesz}.
\end{proof}

\begin{lem} \label{lem: from m to k}
There exists a constant $c>0$ such that the following holds.
Let $T = T_{\vec{0},0}$.  Then, for $z \in \Z^2$ and $r \geq 2 \norm{z}$,
$$ \Pr_z \br{ T \leq \tau_r } \geq \frac{c \log (r/\norm{z})}{\log r} . $$
\end{lem}

\begin{proof}
Let $a:\Z^2 \to [0,\infty)$ be the \emph{potential kernel}
defined in Chapter 1.6 of \cite{IntersRW}.  That is, $a(0) = 0$, $a(\cdot)$ is harmonic in
$\Z^2 \setminus \set{0}$, and there exist constants $c_1,c_2 > 0$ such that
for any $z \in \Z^2 \setminus \set{0}$,
$a(z) = c_1 \log \norm{z} + c_2 + O(\norm{z}^{-2})$.
Since $a(\cdot)$ is harmonic in $\Z^2 \setminus \set{0}$,
if $r > \norm{z}$ then
$a(S(t))$ is a martingale up to time $T' = \min \set{T,\tau_r}$.  Thus,
$$ a(z) = \sr{ 1 - \Pr_z \br{ T \leq \tau_r } } \cdot \E_z \Br{ a(S(T')) }{ T > \tau_r } , $$
which implies
\[ \Pr_z \br{ T \leq \tau_r } \geq 1 - \frac{c_1 \log \norm{z} + c_2 + O(\norm{z}^{-2}) }{
c_1 \log r + c_2 + O(r^{-2}) } .\qedhere \]
\end{proof}

We also need an upper bound,

\begin{lem} \label{lem: from z to hit r before R}
There exists a constant $C > 0$ such that for every $z \in \Z^2$ and $r,R$
such that $1 \leq r \leq \frac{1}{2} \norm{z} \leq \frac{1}{4} R$,
$$ \Pr_z \br{ T_r \leq \tau_R } \leq C \cdot \frac{\log (R/ \norm{z})}{\log (R/r)} . $$
\end{lem}

\begin{proof}
Using the potential kernel from the proof of Lemma \ref{lem: from m to k}
with the stopping time $\min \set{ T_r , \tau_R }$,
there exists a constant $c_1 > 0$ such that
\begin{align*}
\Pr_z \br{ T_r \leq \tau_R } & \leq
\frac{c_1 ( \log R - \log \norm{z} ) + O(R^{-1} + \norm{z}^{-2}) }
{ c_1 ( \log R - \log r ) + O(r^{-2}) } \\
& \leq C \cdot \frac{\log (R/ \norm{z})}{\log (R/r)},
\end{align*}
for some constant $C > 0$.
\end{proof}

\begin{lem} \label{lem: from z to alpha z}
For any $0 < \alpha < 1$, there exists a constant $C >0$ such that the following holds.
Let $z \in \Z^2$ such that $\norm{z} \geq 1/\alpha$.  Then for any $n \in \N$ such that
$n > \norm{z}^4$,
$$ \Pr_z \br{ T_{\alpha \norm{z}} \geq n } \leq \frac{C}{\log (n / \norm{z}^4) } . $$
\end{lem}

\begin{proof}
By adjusting the constant, we can assume without loss of generality that $n/\norm{z}^4$ is large enough.
Let $r = \alpha \norm{z}$ and $R = n^{1/4}$.
Using the potential kernel from the proof of Lemma \ref{lem: from m to k}
with the stopping time $T' = \min \set{ T_r , \tau_R }$,
\begin{align} \label{eqn: T(r) geq tau(R)}
\Pr_z \br{ T_r \geq \tau_R } & \leq \frac{c_1 \log (\norm{z}/r) + O(r^{-1}) }{c_1 \log (R/r) + O(r^{-1}) }
\leq \frac{C_1}{\log (n/\norm{z}^4) } ,
\end{align}
for some constant $C_1 = C_1(\alpha) > 0$ independent of $z$ and $n$.
Also, considering the martingale $\norm{S(t)}^2-t$ up to time $\tau_R$
shows that
$\E_z \br{ \tau_R } \leq (R+1)^2$.
Thus, by Markov's inequality,
\begin{align} \label{eqn: decay of exit prob.}
\Pr_z \br{ \tau_R > n } & \leq \frac{4}{\sqrt{n}} .
\end{align}
\eqref{eqn: T(r) geq tau(R)} and \eqref{eqn: decay of exit prob.}
together prove the proposition, since
\[ \Pr_z \br{ T_r \geq n } \leq \Pr_z \br{ T_r \geq \tau_R } + \Pr_z \br{ \tau_R > n } .
\qquad \qedhere \]
\end{proof}

\begin{lem} \label{lem: sqrt(n) to r}
There exists a constant $C>0$ such that for all $n \in \N$
and $1 \leq r \leq \frac{1}{2} \sqrt{n}$ the following holds.
Let $z \in \Z^2$ be such that $\norm{z} \geq \sqrt{n}$.  Then,
$$ \Pr_z [ T_r \leq n ] \leq \frac{C}{\log (n/r^2)} . $$
\end{lem}

\begin{proof}
For $m \geq 1$, let $A_m$ be the event
$\set{ \tau_{m \norm{z}} < T_r \leq \tau_{(m+1) \norm{z} } \leq n }$.
The family $\set{A_m}$ consists of pairwise disjoint events, and
$$\Pr_z [ T_r \leq n ] \leq \sum_{m = 1}^{\infty} \Pr [ A_m ] . $$
For every $m \geq 1$, using the strong Markov property at time $\tau_{m \norm{z}}$,
\begin{align*}
\Pr_z [A_m] & \leq \Pr_z [ \tau_{m \norm{z} } \leq n ] \cdot
\max \set{ \Pr_x [T_r \leq \tau_{(m+1) \norm{z} } ] \ : \
m \norm{z} \leq \norm{x} \leq m \norm{z} + 1 } .
\end{align*}
By Lemma \ref{lem: maximum of RW}, there exist constants $C_1, c_2 > 0$ such that
\begin{align*}
\Pr_z [ \tau_{m \norm{z} } \leq n ] & \leq
\Pr_z \Big[ \max_{1 \leq j \leq n} \norm{S(j)} \geq m \norm{z} - \norm{z} \Big]
\leq C_1 \exp \sr{ - c_2 m^2 } .
\end{align*}
By Lemma \ref{lem: from z to hit r before R},
for any $x \in \Z^2$ such that $m \norm{z} \leq \norm{x} \leq m \norm{z} +1$,
\begin{align*}
\Pr_x [T_r \leq \tau_{(m+1) \norm{z} } ]
\leq \frac{c_3}{ \log (n/r^2) } ,
\end{align*}
for some constant $c_3 > 0$. Summing over all $m \geq 1$,
\[ \Pr_z [ T_r \leq n ] \leq \frac{c_3}{ \log (n/r^2) }
\sum_{m=1}^{\infty} c_1 \exp \sr{ - c_2 \cdot m^2 } . \qedhere\]
\end{proof}

\subsubsection{Upper bound in two dimensions}

For $z \in \Z^2$ and $k \in \N$, let $Q(z,k) =
\set{ z+(j,j')  \ : \ -k \leq j,j' \leq k }$; i.e., $Q(z,k)$
is the square of side length $2k+1$ centered at $z$.
For a path $x(0),x(1),\ldots,x(n)$ in $\Z^2$, we denote by
$x[s,t]$ the path $x(s),x(s+1),\ldots,x(t)$.

\begin{lem} \label{lem: Prob R hits Q}
There exist constants $c,C>0$ such that for all $n,k \in \N$
such that $k \leq n^{1/4}$, and all $z \in \Z^d$ such that $\norm{z} \geq 5 \sqrt{n}$,
$$ \Pr \br{ R(n) \cap Q(z,k) \neq \emptyset } \leq
\frac{C}{\log n} \cdot \exp \Big( - c \frac{\norm{z}^2}{n} \Big) . $$
\end{lem}

\begin{proof}
Let $\lambda = \norm{z} - 2 \sqrt{n}$.
Let $T$ be the first time the walk $S(\cdot)$ started at $0$ hits $Q(z,k)$.
Then $\tau_{\lambda} < T_{z,2k} < T$.
By Lemmas \ref{lem: maximum of RW} and \ref{lem: sqrt(n) to r},
\begin{align*}
\Pr \br{ R(n) \cap Q(z,k) \neq \emptyset } & \leq \Pr [ \tau_{\lambda} \leq n ] \cdot
\max \set{ \Pr_x [ T_{z,2k} \leq n ] \ : \ \lambda \leq \norm{x} \leq \lambda + 1 } \\
& \leq \Pr \Big[ \max_{1 \leq j \leq n} \norm{S(j)} \geq \lambda \Big] \cdot \frac{c_1}{\log n} \\
& \leq \frac{c_2}{\log n} \cdot \exp \Big( - c_3 \frac{\norm{z}^2}{n} \Big) ,
\end{align*}
for some constants $c_1,c_2,c_3>0$.
\end{proof}

\begin{lem} \label{lem: Prob R hits near Q}
There exists a constant $C>0$ such that the following holds.
For all $n,k \in \N$ such that $k \leq n^{1/4}$, and all $z \in \Z^d$ such that $1 \leq \norm{z} < 5 \sqrt{n}$,
$$ \Pr \br{ R(n) \cap Q(z,k) \neq \emptyset } \leq
C \cdot \frac{\log (10\sqrt{n} / \norm{z})}{\log n} . $$
\end{lem}

\begin{proof}
By adjusting the constant, we can assume without loss of generality that $\norm{z} \geq 3k$.
Let $Q = Q(z,k)$.
Define $\sigma_0 = 0$, and for $i \geq 1$, define
$$ \sigma_{i} = \tau_{10^i \sqrt{n}} = \inf \set{ t \geq 0 \ : \ \norm{S(t)} \geq 10^i \sqrt{n} }.$$
The event
$\set{ R(n) \cap Q \neq \emptyset}$ is contained in the event
$$\set{S[0,\sigma_1] \cap Q \neq \emptyset} \cup
\bigcup_{i \geq 1} \set{S[\sigma_i,\sigma_{i+1}] \cap Q \neq \emptyset \ , \ \sigma_i \leq n} . $$

Since $3k \leq \norm{z} < 5 \sqrt{n}$, we have that the event
$\set{ S[0,\sigma_1] \cap Q \neq \emptyset }$ implies that the
random walk started at $0$ hits the ball of radius $2k$ around $z$
before exiting the ball of radius $20 \sqrt{n}$ around $z$.
Translating by minus $z$ we get by Lemma~\ref{lem: from z to hit r
before R}
that there exists a constant $C_1 > 0$ such that
$$ \Pr \br{ S[0,\sigma_1] \cap Q \neq \emptyset } \leq
\Pr_{-z} \br{ T_{2k} \leq \tau_{20 \sqrt{n}} }
\leq C_1 \cdot \frac{\log (10 \sqrt{n} / \norm{z})}{\log n} . $$
Fix $i \geq 1$. By Lemma~\ref{lem: maximum of RW},
$$ \Pr [ \sigma_i \leq n] \leq \Pr \Big[ \max_{0 \leq j \leq n} \norm{S(j)} \geq 10^{i} \sqrt{n} \Big]
\leq C_2 \cdot \exp \big( - C_3 \cdot 10^{2i} \big) ,$$
for some constants $C_2,C_3 > 0$.
Using Lemma~\ref{lem: from z to hit r before R} again,
\begin{align*}
\Pr [ S[ \sigma_i, \sigma_{i+1} ] \cap Q \neq \emptyset \ | \ \sigma_i \leq n] \leq
 \frac{C_4}{\log n} ,
\end{align*}
for some constant $C_4 > 0$.
Therefore,
\[
\Pr [ R(n) \cap Q \neq \emptyset ] \leq C_1 \cdot \frac{\log (10 \sqrt{n} / \norm{z})}{\log n}
+ \frac{C_2 \cdot C_4}{\log n} \sum_{i \geq 1} \exp \big( - C_3 \cdot 10^{2i} \big).\qedhere
\]
\end{proof}

We have reached the main geometric lemma,

\begin{lem} \label{lem: boundary in square}
There exists a constant $C>0$ such that the following holds.
Let $n,k \in \N$, let $Q = Q(0,k)$ and let $z \sim Q$.
Then,
$$ \Pr_z \br{ \p R(n) \cap Q \neq \emptyset } \leq C \cdot \frac{\log^2 k}{\log n} . $$
\end{lem}

\begin{proof}
Without loss of generality assume that $\log^2 k \leq \log n$.
Define $Q^+ = Q(0,k+1)$. So
$Q^+$ contains the union of $Q$ with all vertices that are adjacent to $Q$.
Define $\tau_0 = 0$, and inductively
$$ \sigma_j = \inf \set{ t \geq \tau_j \ : \ \norm{S(t)} \geq 10 k} , $$
$$ \tau_{j+1} = \inf \set{ t \geq \sigma_j \ : \ S(t) \in Q^+ } . $$
If $Q^+ \subseteq R(n)$ then $\p R(n) \cap Q = \emptyset$.
Thus, it suffices to upper bound the probability of the event $\set{Q^+ \not\subset R(n)}$.
With hindsight choose 
$m = \lceil \log k \cdot \log n \rceil$.
Set $V_j = \big\{ \sigma_{j+1} - \sigma_j \geq \frac{n}{2m} \big\}$
and $U_j = \set{ Q^+ \not\subset R(\sigma_j) }$.
We prove the following inclusion of events
\begin{align} \label{eqn: event inclusion}
\set{ Q^+ \not\subset R(n) } \subseteq \set{ \sigma_0 \geq n/2 } \cup
U_m \cup \bigcup_{j=0}^{m-1} (U_j \cap V_j) .
\end{align}
Assume that the event on the right-hand side of \eqref{eqn: event inclusion} does not occur;
i.e., assume that $\sigma_0 < n/2$, that $\ov{U_m}$,
and that for all $0 \leq j \leq m-1$, $\ov{U_j} \cup \ov{V_j}$.
Let $J = \min \set{ 0 \leq j \leq m \ : \ \ov{U_j} }$.
Consider the following cases:
\begin{itemize}
\item Case 1: $J=0$. Then $Q^+ \subset R(\sigma_0)$.
Since $\sigma_0 < n/2$, we get that $Q^+ \subset R(n)$.

\item Case 2: $J>0$.
Since we assumed that $\ov{U_m}$, we know that $1 \leq J \leq m$.
By the assumption
$\cap_{j=0}^{m-1} (\ov{U_j} \cup \ov{V_j} )$,
we have that
$\sigma_{j+1} - \sigma_j < n/2m$,
for all $0 \leq j \leq J-1$.
Since we assumed that
$\sigma_0 < n/2$, we get that
$$ \sigma_J = \sigma_0 + \sum_{j=0}^{J-1} \sigma_{j+1} - \sigma_j < n . $$
But $J$ was chosen so that $\ov{U_J}$ occurs, so
$Q^+ \subset R(\sigma_J) \subset R(n)$.
\end{itemize}
This proves \eqref{eqn: event inclusion}.

Fix $j \geq 0$.
The martingale $\norm{S(t)-z}^2-t$ shows that $\E_z [\sigma_j - \tau_j \ | \ \F(\tau_j) ] \leq C_1 k^2$
for some constant $C_1>0$.  Using Markov's inequality,
\begin{align} \label{eqn: sigma-tau geq n/2m}
\Pr_z \Big[ \sigma_{j} - \tau_{j} \geq \frac{n}{4m} \ \Big| \ \F(\tau_{j}) \Big] \leq \frac{C_2 m k^2}{n} ,
\end{align}
for some constant $C_2>0$.
By Lemma \ref{lem: from z to alpha z}, there exists a constant $C_3>0$ such that
\begin{align} \label{eqn: tau-sigma}
\Pr_z \Big[ \tau_{j+1} - \sigma_j \geq \frac{n}{4m}  \ \Big| \  \F(\sigma_j) \Big] \leq \frac{C_3}{\log n} .
\end{align}
The two inequalities, \eqref{eqn: sigma-tau geq n/2m} and \eqref{eqn: tau-sigma}, imply that
\begin{align} \label{eqn: sigma - sigma}
\Pr_z \Br{ V_j }{ \F(\sigma_j) } \leq \frac{C_4}{\log n} ,
\end{align}
for some constant $C_4 > 0$.
Using Lemma \ref{lem: from m to k},
there exists a universal constant $C_5 >0$ such that for any $x \in Q^+$,
$$ \Pr_z \Br{ x \in S[\tau_j, \sigma_j] }{ \F(\tau_j) } \geq \frac{C_5}{\log k} . $$
Thus,
\begin{align} \label{eqn: Q not subset of R(sigma j)}
\Pr_z [ U_j ]  = \Pr_z \br{ Q^+ \not\subset R(\sigma_j) } & \leq
\min \big\{ 1,|Q^+| \cdot ( 1- C_5 /\log k )^{j+1} \big\} \nonumber \\
& \leq \min \big\{ 1, C_6 k^2 \exp ( - C_5 (j+1) / \log k ) \big\} ,
\end{align}
for some constant $C_6>0$.
Plugging \eqref{eqn: sigma-tau geq n/2m},
\eqref{eqn: sigma - sigma} and \eqref{eqn: Q not subset of R(sigma j)}
into \eqref{eqn: event inclusion} yields
\begin{align} \label{eqn: Q contains boundary}
\Pr_z \br{ Q^+ \not \subset R(n) }
& \leq \Pr_z \br{ \sigma_0 \geq n/2 } + \Pr_z \br{ U_{m} }
+ \sum_{j=0}^{K} \Pr_z \br{ U_j \ , \ V_j } + \sum_{j > K} \Pr_z \br{ U_j \ , \ V_j } \nonumber \\
& \leq C_7 \Big( \frac{k^2}{n} + n^{-C_8}
+ \sum_{j=0}^K \frac{1}{\log n} + \sum_{j > K} \frac{k^2 \exp ( - C_5 (j+1) / \log k )}{\log n} \Big) \nonumber \\
& \leq \frac{C_9 \log^2 k}{\log n} ,
\end{align}
where $K = \lceil 4 \log^2 k / C_5 \rceil$ and $C_7,C_8,C_9 > 0$ are constants.
\end{proof}

\begin{dfn}
Define $\Lambda(k) = \set{ (2k+1)z \ : \ z \in \Z^2 }$. The
collection $\set{Q(z,k)}_{z \in \Lambda(k)}$ consists of disjoint squares that cover $\Z^2$.
For $k , n \in \N$ and $z \in \Z^2$,
define $I(z,k,n)$ to be the indicator function of the event $\set{ \p R(n) \cap Q(z,k) \neq \emptyset }$.
Define $$M(k,n) = \sum_{z \in \Lambda(k)} I(z,k,n) ,$$
the number of squares that intersect $\p R(n)$.
\end{dfn}

\begin{lem} \label{lem: number of squares with boundary}
There exists a constant $C>0$ such that for every $k,n \in \N$,
$$ \E \br{ M(k,n) } \leq
C \cdot \max \Big\{ 1 , \frac{n}{k^2} \cdot \frac{\log^2 k}{\log^2 n}  \Big\} .  $$
\end{lem}

\begin{proof}
Fix $k,n \in \N$.
For $z \in \Z^2$, the event $\set{\p R(n) \cap Q(z,k) \neq \emptyset}$ implies the event
$$\Big\{\max_{0 \leq j \leq n} \norm{S(j)} \geq \norm{z} - \sqrt{2} (k + 1) \Big\}.$$
We start with an a-priori bound. Using Lemma~\ref{lem: maximum of E[RW]},
\begin{align*}
\E \br{ M(n,k) } & \leq \sum_{z \in \Lambda(k)}
\Pr \Big[ \norm {z} \leq \max_{0 \leq j \leq n} \norm{S(j)} + \sqrt{2} (k + 1) \Big] \\
& \leq C_1 \cdot \max \set{ 1, k^{-2}  \cdot \E \Big[ \max_{0 \leq j \leq n} \norm{S(j)}^2 \Big] } \\
& \leq C_2 \cdot \max \set{1 , \frac{n}{k^2} } ,
\end{align*}
for some constants $C_1,C_2>0$.
Thus, we can assume without loss of generality that $k < k+1 \leq (n-\sqrt{n})^{1/4} \leq n^{1/4}$.

Let
$$ \tau_Q(z) = \inf \set{ t \geq 0 \ : S(t) \in Q(z,k+1) }  $$
and let
$$J(z,k,n) = \1{ \tau_Q(z) \leq n - \sqrt{n} } \cdot I(z,k,n).$$
For all $z \in \Lambda(k)$, a.s.
$$ I(z,k,n) \leq \1{ n- \sqrt{n} < \tau_Q(z) \leq n } + J(z,k,n) . $$
Summing over all $z \in \Lambda(k)$, a.s.
\begin{align} \label{eqn: M is sum}
 M(n,k) \leq 4 \sqrt{n} + \sum_{z \in \Lambda(k)} J(z,n,k) .
\end{align}
By the strong Markov property at time $\tau_Q(z)$ and Lemma \ref{lem: boundary in square},
there exists a constant $C_3>0$ such that a.s.
\begin{align} \label{eqn: r and q given tau}
\Pr \Br{ \p R(n) \cap Q(z,k) \neq \emptyset }{ \tau_Q(z) \leq n-\sqrt{n} } &
\leq C_3 \cdot \frac{\log^2 k }{\log n} .
\end{align}
By Lemma~\ref{lem: Prob R hits near Q},
there exists a constant $C_4>0$ such that for all $z \in \Z^d$ with $1 \leq \norm{z} < 5 \sqrt{n}$,
$$ \Pr \br{ \tau_Q(z) \leq n - \sqrt{n} } \leq C_4 \cdot \frac{\log (10 \sqrt{n} / \norm{z})}{\log n} , $$
which implies
\begin{align} \label{eqn: Pr[J(z,n,k)] leq}
\Pr [ J(z,k,n) ] \leq C_5 \cdot \frac{\log^2 k }{\log n} \cdot \frac{\log (10 \sqrt{n} / \norm{z})}{\log n},
\end{align}
for some constant $C_5>0$.

Denote $\Gamma = 5 \sqrt{n} / (2k+1)$.
Summing over all $z \in \Lambda(k)$ such that $2 \leq \norm{z} < 5 \sqrt{n}$,
\begin{align} \label{eqn: sum for first part of M}
\nonumber \sum_{\substack{ z \in \Lambda(k) \\ 2 \leq \norm{z} < 5\sqrt{n} } } & \log (10 \sqrt{n} / \norm{z})
\leq \sum_{\substack{ x,y \in \Z \\ 2 \leq x^2+y^2 < \Gamma^2 } }  \log (2 \Gamma / \sqrt{x^2+y^2} ) \nonumber \\
& \leq C_6 \Gamma \sum_{2 \leq x \leq \Gamma } \log (2 \Gamma / x)
 \leq C_7 \Gamma^2 ,
\end{align}
for some constants $C_6,C_7 > 0$.
Plugging \eqref{eqn: sum for first part of M} into \eqref{eqn: Pr[J(z,n,k)] leq}, and
summing over all $z \in \Lambda(k)$ such that $\norm{z} < 5 \sqrt{n}$, we get
\begin{align} \label{eqn: first part of M}
\sum_{z \in \Lambda(k) : \norm{z} < 5\sqrt{n}} & \Pr[ J(z,k,n) ]
\leq C_8 \cdot \frac{\log^2 k }{\log^2 n} \cdot \frac{n}{k^2} ,
\end{align}
for some constant $C_8>0$.
In addition, by Lemma~\ref{lem: Prob R hits Q},
there exist constants $C_9,C_{10}>0$ such that for every $z \in \Lambda(k)$ such that $\norm{z} \geq 5 \sqrt{n}$,
$$ \Pr \br{ \tau_Q(z) \leq n - \sqrt{n} } \leq
\frac{C_9}{\log n} \cdot \exp \Big( - C_{10} \frac{\norm{z}^2}{n} \Big) , $$
which implies, using \eqref{eqn: r and q given tau},
\begin{align*}
\Pr [ J(z,k,n) ]
& \leq C_{11} \cdot \frac{\log^2 k }{\log^2 n} \cdot \exp \Big( - C_{10} \frac{\norm{z}^2}{n} \Big) ,
\end{align*}
for some constant $C_{11}>0$.
Summing over all $z \in \Lambda(k)$ such that $\norm{z} \geq 5 \sqrt{n}$,
\begin{align} \label{eqn: M second}
\nonumber \sum_{z \in \Lambda(k) : \norm{z} \geq 5\sqrt{n}} \Pr[ J(z,k,n) ]
& \leq C_{11} \cdot \frac{\log^2 k }{\log^2 n} \sum_{z \in \Lambda(k) : \norm{z} \geq 5\sqrt{n}}
 \exp \Big( - C_{10} \frac{\norm{z}^2}{n} \Big) \\
& \leq C_{12} \cdot \frac{\log^2 k }{\log^2 n} \cdot \frac{n}{k^2},
\end{align}
for some constant $C_{12} > 0$.
The lemma follows by \eqref{eqn: M is sum}, \eqref{eqn: first part of M} and \eqref{eqn: M second}.
\end{proof}

For $k < n \in \N$, let $\p(k,n)$ be the vector $(I(z,k,n) )_{z \in \Lambda(k) \cap [-2n,2n]^2}$.
Note that
$$ M(k,n) = \sum_{z \in \Lambda(k)} I(z,k,n) = \sum_{z \in \Lambda(k) \cap [-2n,2n]^2} I(z,k,n) . $$

\begin{lem} \label{lem: entropy p(k,n) given p(k',n)}
Let $k,\ell, n \in \N$ and let $k' = (2 \ell+1)k+\ell$.
Then,
$$ H(\p(k,n) \ | \ \p(k',n)) \leq \E [M(k',n)] \cdot (2 \ell +1)^2 . $$
\end{lem}

\begin{proof}
For any $z' \in \Lambda(k')$, the square $Q(z',k')$ is of side length
$2k'+1 = (2\ell+1)(2k+1)$, and so $Q(z',k')$ can be tiled by $(2\ell+1)^2$ disjoint squares
from the collection $\set{ Q(z,k) }_{z \in \Lambda(k)}$.

If $Q(z,k) \subset Q(z',k')$, then $I(z,k,n) \leq I(z',k',n)$.  Thus,
conditioned on the vector $\p(k',n)$, there are at most $2^{M(k',n) \cdot (2\ell+1)^2 }$
possibilities for the vector $\p(k,n)$.  By clause \ref{item: H leq log} of Proposition
\ref{prop: properties of entropy}, and by the definition of conditional entropy,
$H(\p(k,n)\ |\ \p(k',n)) \leq \E [M(k',n) \cdot (2 \ell+1)^2]$.
\end{proof}

\begin{lem} \label{lem: upper bound}
There exists a constant $C_2 > 0$ such that for all $n$,
$$ H(R(n)) \leq C_2 \frac{n}{\log^2 (n)} . $$
\end{lem}

\begin{proof}
Since the vector $\p(0,n)$ determines $R(n)$,
clauses \ref{item: H(f(X)|X)} and \ref{item: chain cond. entropy} of Proposition~\ref{prop: properties of entropy}
yield that $H(R(n)) \leq H(\p(0,n))$.

Set $k_0 = 0$, and for $j \geq 0$, define inductively
$k_{j+1} = 3 k_j + 1$.
For every $j \geq 1$, since $3 k_j \leq k_{j+1} \leq 4 k_j$,
it holds that $\frac{\log k_{j}}{k_{j}} \leq 9 j 3^{-j}$.
Let $m>0$ be the smallest $j$ such that $k_j > n$.
The entropy of $\p(k_m,n)$ is zero.
By Lemmas~\ref{lem: number of squares with boundary} and~\ref{lem: entropy p(k,n) given p(k',n)},
for $0 \leq j \leq m-1$, there exist universal constants $c_2,c_3>0$ such that
\begin{align*}
  H(\p(k_j,n) \ | \ \p(k_{j+1},n) ) 
 \leq c_3 \cdot \max \Big\{ 1, \frac{n}{\log^2 n} \cdot \frac{(j+1)^2}{9^{j+1}} \Big\} .
\end{align*}
Using clause \ref{item: chain cond. entropy} of Proposition \ref{prop: properties of entropy},
there exists a constant $C>0$ such that
\[ H(\p(0,n)) \leq
  \sum_{j=0}^{m-1} H(\p(k_j,n) \ | \ \p(k_{j+1},n) ) + H(\p(k_m,n)) \leq
  C \cdot \frac{n}{\log^2 n} .\qedhere
\]
%
%
%
\end{proof}

\begin{rem}
The proof of Lemma \ref{lem: upper bound}
shows that provided one can calculate the different conditional probabilities
(e.g., with unlimited computational power), one can sample the range of a random walk
using only order $n/\log^2 n$ bits.
\end{rem}

\section{Concluding Remarks and Problems for Further Research}

\subsection{Extracting Entropy}
\label{scn: extracting entropy}

Lemma \ref{lem: entropy lower bound} shows that the entropy of $R(n)$ in two dimensions
is at least $c_2 n/\log^2 n$.  It is interesting to note that one can extract order of $n/\log^2 n$
almost uniformly distributed random bits,
by observing a sample of the range.  We sketch the construction.

Consider the two configurations that appear in Figure \ref{fig: configurations}.
Symmetry implies that conditioned on outside of the configuration, both have the same probability
of occurring.  Thus, any occurrence of such a configuration in the range of the random walk gives
an independent bit, e.g., setting the bit to be $1$ if the right configuration occurs,
and $0$ if the left configuration occurs.
Considerations similar to those raised in the proofs above show that the expected
number of such configurations is of order $n/\log^2 n$.

\begin{figure}[htbp]
\centering
\includegraphics[clip=true]{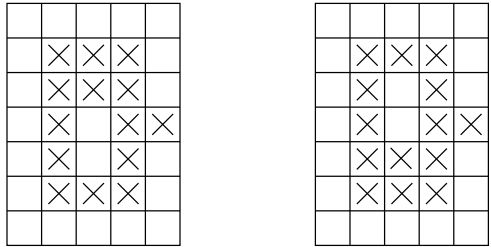}
\caption{Two symmetric configurations.  X's are vertices occupied
by the range.} \label{fig: configurations}
\end{figure}


\subsection{Intersection Equivalence}

Consider the $n \times n$ square around $0$ in $\Z^2$,
and consider the following procedure.  Divide the square into $4$ squares
of side length $n/2$.  Retain each of the squares with probability $1/2$,
independently.
Continue inductively: at level $k$, divide each remaining square of side length
$n 2^{-(k-1)}$ into $4$ squares of side length $n 2^{-k}$, and retain
each one with probability $k/(k+1)$ independently.

This procedure produces a random subset of the $n \times n$ square, denote this set by $Q(n^2)$.
In \cite{Peres}, Peres shows that
the sets $Q(n^2)$ and $R(n^2)$ are {\em intersection equivalent};
that is, there exist constants $c,C>0$ such that for any set $A \subset \Z^2$,
$$ c \leq \frac{\Pr [ Q(n^2) \cap A \neq \emptyset ] }{ \Pr [ R(n^2) \cap A \neq \emptyset ] } \leq C . $$
The entropy $H(Q(n^2))$ is of order $n^2 / \log^2(n)$, as is $H(R(n^2))$.
Note that intersection equivalence does not imply or follow from equal entropy.
See \cite{Peres} for more details.

\subsection{Open Questions}

Let $G$ be an infinite graph, and let $\set{S(n)}_{n \geq 0}$ be a simple random walk
on $G$.  Let $R(n) = \set{S(0), S(1), \ldots, S(n)}$ be the range of the walk at time $n$.
Let $H(n)$ be the entropy of $R(n)$.

Our results above suggest the following natural questions.

\begin{itemize}

\item Assume $G$ is vertex transitive (that is, for any two vertices $x,y$
there exists an automorphism of $G$ taking $x$ to $y$).  Is it true that if
$S(\cdot)$ is transient then $H(n)$ grows linearly in $n$?
It is not difficult to produce examples of non-transitive graphs, that are transient but have
sub-linear entropy.

\item How small can $H(n)$ be in transient graphs?  It is possible to construct
(spherically symmetric) trees that are transient but have $H(n) = O(\log^2 n)$.
Is it possible to get a smaller entropy?




\end{itemize}


\bibliographystyle{acm}


\end{document}